\documentclass{article}
\usepackage{amssymb,amsfonts,amsmath,amsthm,tikz}
\usepackage{xcolor}
\usepackage[utf8]{inputenc}
\usepackage{bm}
\usepackage{mathtools}
\usepackage{soul}
\usepackage{enumerate}
\usepackage[english]{babel}
\usepackage{amsthm}
\usepackage{faktor}
\usepackage{enumitem}
\usepackage{hyperref}
\usepackage{mathabx}
\usepackage{mathbbol}
\usepackage{flexisym}

\newcommand{\id}{\mathrm{id}}

\newcommand{\R}{\mathbb{R}}
\newcommand{\N}{\mathbb{N}}

\newcommand{\Ge}{\mathcal{G}_{\infty}}
\newcommand{\Ho}{\mathcal{H}_0}

\newcommand{\LO}{\mathcal{LO}}

\theoremstyle{definition}
\newtheorem{question}{Question}
\usepackage[lite]{amsrefs}

\newcommand{\hr}{\mathrm{Homeo}_{+}{(\R)}}
\theoremstyle{definition}

\newtheorem{theorem}{Theorem}
\newtheorem{lemma}[theorem]{Lemma}
\newtheorem{proposition}{Proposition}
\newtheorem{corollary}[theorem]{Corollary}
\theoremstyle{remark}

\theoremstyle{plain}

\theoremstyle{definition}

\theoremstyle{remark}

\begin{document}

	\title{Classifying the closure of standard orderings on $\mathrm{Homeo}_{+}{(\mathbb{R})}$}
	
	\author{Kyrylo Muliarchyk}
	
	\date{}
	
	\maketitle
	
	\begin{abstract}
		We characterize a closure of the set of dynamical-lexicographic orderings on $\hr$ and prove the existence of orderings outside of it.
	\end{abstract}
	
	\section{Introduction}
	
	\label{sec:1}
	
	A group $G$ is said to be \emph{left-orderable} if it admits a total ordering relation $\preceq$ that is invariant under left multiplication, in addition to the usual order properties. In other words,
\[
f \prec g \;\; \implies \;\; hf \prec hg \quad \text{for all } f,g,h \in G.
\]

Following Sikora \cite{Sikora}, we equip the set $\LO(G)$ of all left-orderings on $G$ with a topology as follows. For a subset $S \subset G$, let $U_S$ denote the set of all orderings in which each element of $S$ is positive (i.e., greater than the identity). The collection of all such sets $U_S$, with $S$ finite, forms a subbasis for a topology on $\LO(G)$.

An important example of a left-orderable group is $\hr$, the group of all orientation-preserving homeomorphisms of the real line. It is well known (see, e.g., \cite{DNR}*{Proposition 1.1.8}) that every countable left-orderable group acts faithfully on the real line by orientation-preserving homeomorphisms. Equivalently, a countable left-orderable group $G$ admits a canonical embedding
\[
\rho: G \to \hr,
\]
called the \emph{dynamical realization} of $G$. This embedding is uniquely determined up to topological conjugacy.

Left-orderability of $\hr$ is usually established as follows (see, e.g., \cite{ClayRolfsen}{[Example 1.11]}). 
Let $\mathbf{x}=(x_n)_{n=1}^{\infty}$ be a sequence dense in $\R$, and let $\Omega:\N \to \{+,-\}$ be a sign function. 
The \emph{standard dynamical-lexicographic ordering}, or simply \emph{standard}, 
$\preceq = \preceq_{\mathbf{x},\Omega}$ on $\hr$ is defined by declaring $f \succ g$ whenever 
\[
f(x_n) > g(x_n) \text{ and } \Omega(n)=+ \quad \text{or} \quad f(x_n) < g(x_n) \text{ and } \Omega(n)=-,
\] 
where $n$ is the smallest index such that $f(x_n) \neq g(x_n)$.

More generally, given a well-ordering $\leq_w$ of $\R$ and a sign function $\Omega:\R \to \{+,-\}$, the \emph{general dynamical-lexicographic ordering} $\preceq_{\leq_w,\Omega}$ is defined by comparing $f,g \in \hr$ at the $\leq_w$-least point $x \in \R$ where $f(x)\neq g(x)$. Clearly, the standard orderings form a special case of general dynamical-lexicographic orderings—namely, those for which the initial segment of the associated well-ordering is a dense subset of $\R$.

We denote by $\mathcal{S}$ the set of all standard orderings.
%and by $\overline{\mathcal{S}}$ its closure in $\LO(\hr)$. 
It is natural to ask (see \cite{DNR}*{Example 2.2.2]}) whether there exist non-standard orderings on $\hr$. 
The answer is negative: any general dynamical-lexicographic ordering in which the elements indexed by the initial segment $\omega$ do not form a dense subset of $\R$ is necessarily non-standard. 
For example, one may take dense sequences $\mathbf{x}=(x_n)$ and $\mathbf{y}=(y_n)$ of positive and negative numbers, respectively, and let the well-ordering begin with $\mathbf{x}$ followed immediately by $\mathbf{y}$.

In this paper, we investigate the following question.  
\begin{question}\label{question}
Are the standard orderings dense in $\LO(\hr)$?   
\end{question}
Specifically, our goal is to characterize the sequential closure $\mathrm{sc}(\mathcal{S})$ of the set of all standard orderings, as well as its topological closure $\overline{\mathcal{S}}$.

We prove.
\begin{theorem}\label{seqcl}
    The sequential closure $\mathrm{sc}(\mathcal{S})$ of the set of all standard orderings is the set $\mathcal{G}$ of all general dynamical-lexicographic orderings.
\end{theorem}

We begin the characterization of the set $\overline{\mathcal{S}}$ by introducing the following sets.
For a function $f \in \hr$, define
\[
A_f \coloneq \{x \in \R : f(x) > x\}, \qquad 
B_f \coloneq \{x \in \R : f(x) < x\}.
\]

Geometrically, $A_f$ is the set of points where the graph of $f$ lies \textbf{above} the diagonal $y=x$, while $B_f$ consists of the points where the graph lies \textbf{below} the diagonal. 

It is straightforward to verify that for any $f \in \hr$, the sets $A_f$ and $B_f$ are open and disjoint subsets of $\R$. 

We call an ordering $\preceq \in \LO(\hr)$ \emph{typical} if the sign of each $f \in \hr$ is determined solely by the sets $A_f$ and $B_f$. In other words, in a typical ordering, two functions $f,g \in \hr$ have the same sign whenever $A_f = A_g$ and $B_f = B_g$. We denote the set of all typical orderings by $\mathcal{T}$.

In Section \ref{sec:proof1}, we establish the following result.

\begin{theorem}\label{S=T}
The closure $\overline{\mathcal{S}}$ of the set of standard orderings coincides with the set $\mathcal{T}$ of all typical orderings.
\end{theorem}

As a consequence, in Section \ref{sec:proof2} we answer Question~\ref{question} negatively by proving the following hierarchy:

\begin{theorem}\label{hierarchy}
The following strict inclusions hold:
\[
\mathcal{S}\subsetneq \mathcal{G} \subsetneq \mathcal{T} \subsetneq \LO(\hr).
\]
\end{theorem}

This paper is organized in the following way:

In Section \ref{sec:top}, we discuss the Sikora topology on $\LO(G)$, and convergence in it specifically for the case of $G=\hr$. We study the dynamical-lexicographic orderings. Specifically, we prove that for any such ordering, the data from only countably many points is relevant. We also prove Theorem \ref{seqcl} and the first inclusion of Theorem \ref{hierarchy}.

In Section \ref{sec:dynre}, we remind the construction of the dynamical realization and its connection to standard orderings.

In Section \ref{sec:proof1}, we prove Theorem \ref{S=T} and the second inclusion of Theorem \ref{hierarchy}.

In Section \ref{sec:proof2}, we finish the proof of Theorem \ref{S=T} by giving the examples that illustrate the strictness of the inclusions.

	\section{Topology on $\LO(\hr)$}

\label{sec:top}

The following criterion of left-orderability is well known (see, e.g. \cite{ClayRolfsen}*{Theorem 1.24}).
\begin{theorem}\label{positivecone}
    A group $G$ is left-orderable if and only if there exists a subset $P\subset G$ such that
    \begin{enumerate}
        \item $P\cdot P\subset P$;
        \item $P\cup P^{-1}=G\setminus\{1\}$.
    \end{enumerate}
\end{theorem}
\begin{proof}[Sketch]
    Suppose $G$ is left-orderable, and let $\preceq$ be a left-ordering on $G$. Then the set 
    \[
    P=P_{\preceq}=\{g\in G\mid g\succ 1\}
    \]
    of all positive elements satisfies the conditions of the theorem.

    Conversely, given such a set $P$, the relation $g\prec h$ if $g^{-1}h \in P$ defines a left-ordering on $G$.  
\end{proof}

A set $P$ satisfying the conditions of Theorem \ref{positivecone} is called a \emph{positive cone}. Positive cones on a given group $G$ are in one-to-one correspondence with left-orderings on $G$ via the map $\preceq \mapsto P_{\preceq}$. This bijection identifies the set $\LO(G)$ of all left-orderings with the set $\mathcal{P}(G)$ of all positive cones in the power set $2^G$. The power set $2^G$ admits a canonical bijection $2^G\to \{0,1\}^G$, where we equip $\{0,1\}^G$ with the product topology. The Sikora topology on $\LO(G)$ coincides with the topology induced from this product topology, since the sets $U_S$ correspond to the cylindrical sets forming the standard basis for the product topology.

\medskip

Consider any general dynamical-lexicographic ordering $\preceq=\preceq_{\leq_w,\Omega}$. We say that a point $x\in \R$ is \emph{relevant} if changing $\Omega(x)$ changes the ordering $\preceq$. Equivalently, $x$ is relevant if there exists a function $f\in\hr$ whose sign is determined by its value at $x$. Denote by $X=X_{\preceq}$ the set of all relevant points. The well-ordering $\leq_w$ on $\R$ induces a well-ordering on $X$, so that $X$ can be treated as a transfinite sequence.

\begin{proposition}\label{relevant}
    A point $x\in\R$ is relevant if and only if it cannot be approximated by the $\leq_w$-preceding points. 
\end{proposition}
\begin{proof}
    Suppose $x$ is a relevant point. Then there exists $f\in\hr$ whose $\preceq$-sign is determined by $f(x)$. In particular, $f(x)\neq x$. By continuity, $f(y)\neq y$ for all $y$ in some neighborhood of $x$. If any such $y$ precedes $x$, then the sign of $f$ would already be determined by $f(y)$, contradicting the assumption. Thus $x$ cannot be approximated by preceding points.

    Conversely, suppose $x$ is not a limit point of $\{y\in \R\mid y<_w x\}$. Then there exists a neighborhood $U=(x-a,x+a)$ disjoint from the set $\{y\in \R\mid y<_w x\}$ of preceding points. Let $f\in\hr$ be supported on $U$, for example
    \[
    f(y)=\begin{cases}
            y+\frac{a}{\pi}\cos\left(\tfrac{y-x}{a}\pi\right), & \text{if } y\in U, \\
            y, & \text{if } y\notin U.
    \end{cases}
    \]
    Then $f(y)=y$ for all $y<_w x$, while $f(x)\neq x$. Hence the sign of $f$ is determined at $x$, so $x$ is relevant.
\end{proof}

To study a general dynamical-lexicographic ordering $\preceq=\preceq_{\leq_w,\Omega}$, it suffices to restrict to the set of relevant points.

\begin{proposition}\label{relevant2}
Let $\preceq_i=\preceq_{\leq_w^{(i)},\Omega_i}$, $i=1,2$, be two general dynamical-lexicographic orderings, and let $X_i=X_{\preceq_i}$ be their sets of relevant points. Then $\preceq_1=\preceq_2$ if and only if $X_1=X_2$ as transfinite sequences and $\Omega_1\mid_{X_1}=\Omega_2|_{X_2}$.
\end{proposition}
\begin{proof}
    The ``if'' direction follows directly from the definition of relevant points.

    For the converse, assume $\preceq_1=\preceq_2$. Suppose, for contradiction, that $\alpha$ is the smallest ordinal such that the corresponding points in $X_1$ and $X_2$ differ. Denote these points by $x_1$ and $x_2$. By Proposition \ref{relevant}, let $U_1\ni x_1$ and $U_2\ni x_2$ be neighborhoods separating $x_1$ and $x_2$ from the preceding points, chosen so that $U_1$ and $U_2$ are disjoint. Let $f_1, f_2$ be functions supported on $U_1, U_2$ respectively, and set $f=f_1f_2^{-1}$ if $\Omega_1(x_1)=\Omega_2(x_2)$, or $f=f_1f_2$ otherwise. Then $f$ has different signs in $\preceq_1$ and $\preceq_2$, a contradiction. Thus $X_1=X_2$. A similar argument shows $\Omega_1\mid_{X_1}=\Omega_2|_{X_2}$.
\end{proof}

In particular, Proposition \ref{relevant2} shows that standard orderings are precisely those whose relevant points form an $\omega$-indexed sequence, justifying the inclusion
\[
\mathcal{S}\subsetneq \mathcal{G}
\]
in Theorem \ref{hierarchy}.

\begin{lemma}\label{relcountable}
    For any general dynamical-lexicographic ordering $\preceq=\preceq_{\leq_w,\Omega}$, the set $X=X_{\preceq}$ of relevant points is countable.
\end{lemma}
\begin{proof}
    Recall that $\R$ is second-countable, and fix any countable basis $\{U_n\}$. Each relevant point $x$ is separated from the set $\{y\mid y<_w x\}$ of preceding elements by some neighborhood $U_{n(x)}\ni x$ from the basis. Clearly, each neighborhood can be used for at most one relevant point. Hence $X$ is at most countable. On the other hand, $X$ cannot be finite since it must be dense in $\R$.
\end{proof}

Let $(\preceq_n)_{n=1}^{\infty}$ be a sequence of orderings on $\LO(\hr)$ converging to some ordering $\preceq$. 
Equivalently, for every $f \in \hr$, the sign of $f$ in $\preceq_n$ eventually stabilizes to its sign in $\preceq$. 
We are particularly interested in the case of standard orderings $\preceq_n=\preceq_{\mathbf{x}^{(n)},\Omega_n}$. 
For simplicity, by passing to a subsequence if necessary, we may assume that each sequence $\mathbf{x}^{(n)}$ consists only of relevant points; in particular, it has no repetitions.

\begin{lemma}\label{stable}
    Let $\preceq_n=\preceq_{\mathbf{x}^{(n)},\Omega_n}$ be a convergent sequence of standard orderings. 
    Then each sequence $\bigl( x^{(n)}_m \bigr)_{n=1}^{\infty}$ is eventually constant. 
    Similarly, each sequence $\left( \Omega_n(m) \right)_{n=1}^{\infty}$ is eventually constant.
\end{lemma}

\begin{proof}
We proceed by induction on $m$.

Fix $k\in \N$ and assume that for all $m<k$, the sequences $x^{(n)}_m$ are eventually constant. 
After discarding finitely many initial terms, we may assume these sequences are constant for all $n$, and denote the constants by $x_m = x^{(n)}_m$, $1 \leq m < k$. 
By construction, each $x_m$ is distinct from every $x^{(n)}_k$.

\emph{Step 1: Stabilization of $\Omega_n(k)$.}  
Consider the test function $f_k$ defined so that $f_k(x_m) = x_m$ for $1 \leq m < k$, and $f_k(x)>x$ elsewhere. 
In particular, $f_k(x^{(n)}_k) > x^{(n)}_k$. 
Thus the $\preceq_n$-sign of $f_k$ is precisely $\Omega_n(k)$. 
Since the sequence of signs converges, the sequence $(\Omega_n(k))$ stabilizes.

\emph{Step 2: Stabilization of $x^{(n)}_k$.}  
Suppose first that $x^{(n)}_k$ takes infinitely many distinct values. 
Then it has a strictly monotone subsequence $(x^{(n_\ell)}_k)_{\ell}$. 
Choose disjoint neighborhoods $U_\ell \ni x^{(n_\ell)}_k$, such that $U_l\cap\{x_1,\dots,x_{k-1}\}=\emptyset$, and let $f_\ell$ be supported on $U_\ell$ as in Proposition \ref{relevant}, with $f_\ell(x)>x$ on $U_\ell$. 
Define the test function
\[
    f(x)=
    \begin{cases}
        f_\ell(x), & x \in U_\ell, \ \ell \ \text{even},\\
        f_\ell^{-1}(x), & x \in U_\ell, \ \ell \ \text{odd},\\
        x, & \text{otherwise}.
    \end{cases}
\]
Then, the $\preceq_{n_{\ell}}$-sign of $f$ is $\Omega_{n_{\ell}}(k)$ if $\ell$ is even, and $-\Omega_{n_{\ell}}(k)$ if $\ell$ is odd.

Since the sequence  $(\Omega_{n_{\ell}}(k))$ is eventually constant, for sufficiently large $\ell$, the signs of $f$ with respect to $\preceq_{n_\ell}$ alternate, contradicting convergence. 
Thus, this case is impossible.

Now suppose $x^{(n)}_k$ takes only finitely many values. 
At least one of these, say $x$, occurs infinitely often. 
Construct a test function $f$ with $f(x_m)=x_m$ for $m<k$, $f(x)>x$, and $f(x^{(n)}_k)<x^{(n)}_k$ whenever $x^{(n)}_k\neq x$. 
This is possible since only finitely many points are involved. 
Then $f$ has different signs depending on whether $x^{(n)}_k=x$ or not. 
Since the signs converge, $x^{(n)}_k$ must equal $x$ for all sufficiently large $n$. 
Hence $x^{(n)}_k$ stabilizes.
\end{proof}

\begin{corollary}\label{limgen}
The limit of a convergent sequence of standard orderings $\preceq_n$ is a general dynamical-lexicographic ordering.    
\end{corollary}

\begin{proof}
We construct the limit ordering $\preceq$ explicitly as a general dynamical-lexicographic ordering $\preceq_{X,\Omega}$ by recursively building a dense transfinite sequence $X$ of relevant points and a sign function $\Omega:X\to \{+,-\}$.

By Lemma \ref{stable}, each sequence $\bigl(x^{(n)}_m\bigr)_{n=1}^{\infty}$ eventually stabilizes. 
Let $x_m$ denote its limit, and set $X_\omega=(x_m)_{m\in \N}$. 
Similarly, each $\Omega_n(m)$ stabilizes, and we define $\Omega(x_m)$ to be this eventual constant.

Now let $\alpha$ be a limit ordinal, and suppose the initial $\alpha$-segment $X_\alpha$ of $X$, and $\Omega:X_\alpha\to\{+,-\}$ have been defined through the recursion. 

If $X_\alpha$ is dense in $\R$, we terminate the recursion and extend $X$ arbitrarily to a well-ordered dense sequence in $\R$, together with an arbitrary extension of $\Omega$. 
This defines the ordering $\preceq=\preceq_{X,\Omega}$.

If $X_\alpha$ is not dense, remove $\overline{X_\alpha}$ from the sequences $\bigl(x^{(n)}_m\bigr)_{m=1}^{\infty}$. 
These sequences remain infinite since they are dense in $\R$. 
By the same argument as in Lemma \ref{stable}, the adjusted sequences $\bigl(x^{(n),\alpha}_m\bigr)_{n=1}^{\infty}$ and $\bigl(\Omega^{\alpha}_n(m)\bigr)_{n=1}^{\infty}$ are eventually constant. 
We define 
\[
x_{\alpha+m-1}=\lim_{n\to\infty} x^{(n),\alpha}_m
\]
and 
\[
\Omega(x_{\alpha+m-1})=\lim_{n\to\infty}\Omega^\alpha_n(m)
\]
to be the corresponding limits.

This process terminates at some countable ordinal, since the sequences contain only countably many points, and at each step of the recursion we remove points from the sequences.

Finally, we check that $\preceq$ is indeed the limit of $(\preceq_n)$. 
Let $f\in\hr$ be nontrivial, and let $x\in X$ be the first point with $f(x)\neq x$. 
Such $x$ exists since $X$ is dense. 
Write $x=x_{\alpha+m-1}$ for some $\alpha,m$. 
For all $\beta<\alpha$, we have $f(x_\beta)=x_\beta$, so removing $X_\alpha$ does not affect the $\preceq_n$-signs of $f$. 
For large $n$, the sign of $f$ is determined at $x^{(n),\alpha}_m=x_{\alpha+m-1}$, since $f(x^{(n),\alpha}_k)=x_{\alpha+k-1}$ for $k<m$. 
Thus the $\preceq_n$-signs of $f$ eventually agree with the $\preceq$-sign, proving $\preceq_n \to \preceq$.
\end{proof}

Theorem \ref{seqcl} then follows from
\begin{proposition}
    Each general dynamical-lexicographic ordering is a limit point of standard orderings.
\end{proposition}
\begin{proof}
    Let $\preceq$ be a dynamical-lexicographic ordering, let $X$ be its transfinite sequence of relevant numbers, and let $\Omega:\R\to \{+,-\}$ be its sign function.

    Since $X$ is countable, there is an expanding sequence $F_1\subset F_2\subset F_3\subset\dots$ of finite sets such that $\bigcup F_n=X$.

    We define a sequence $\mathbf{x}_n=(x_m^{(n)})$ by begining with the elements of $F_n$ ordered the same way they are in $X$, and by extending it arbitrarily with the remaining elements from $X\setminus F_n$. We also define $\Omega_n(m)=\Omega(x_m^{(n)})$.

Then the sequence $(\preceq_{\mathbf{x}_n,\Omega_n})$ of standard orderings converges to the ordering $\preceq$. Indeed, let $f\in\hr$ be a non-trivial function, and let its $\preceq$-sign be determined at $x\in X$. Then, whenever $x\in F_n$, its $\preceq_n$-sign is also determined at $x$: the poins that preced $x$ in $F_n$ also preced $x$ in $X$. So $\preceq$-, and $\preceq$-signs of $f$ agree, since $\Omega$ and $\Omega_n$ agree.    
\end{proof}

	\section{Dynamical realization of left-ordered groups}
\label{sec:dynre}

We recall the classical dynamical orderability criterion for countable groups. For a similar criterion for uncountable groups, we refer to \cite{KopMed}*{Theorem 3.4.1}.
\begin{proposition}\label{dynre}
A countable group $G$ is left-orderable if and only if it acts faithfully on the real line by orientation-preserving homeomorphisms.   
\end{proposition}
\begin{proof}
    The "if" part trivially holds since the group $\hr$ is left-orderable.

    We sketch the "only if" part below:

    Suppose that $G$ is equipped with a left-ordering $\preceq$. Let $\{g_n\}_{n=1}^{\infty}$ be the enumeration of the elements of $G$ with $g_1=1$ being the trivial element. We build the function $t:G\to \R$ recursively in the following way: set $t(1)=0$, assuming $t(g_1),\dots, t(g_n)$ has been already defined, if $g_{n+1}$ is the biggest (respectively smallest) among $g_1,\dots, g_{n+1}$ put $t(g_{n+1})=\max\{t(g_1),\dots,t(g_n)\}+1$ (respectively $t(g_{n+1})=\min\{t(g_1),\dots,t(g_n)\}-1$); if $g_i\prec g_{n+1}\prec g_j$, and no other $g_k$ is between $g_i,g_j$, set $t(g_{n+1})=\frac12 (t(g_i)+t(g_j)$. 

    Note that $G$ acts naturally on $t(G)$ by $g \cdot t(h)=t(gh)$. This action extends continuously on the closure of $t(G)$, and linearly on $\R$. The resulting action completes the proof.
    
\end{proof}

Note that the action constructed in the proof of Proposition \ref{dynre}, known as the dynamical realization of the group $G$, allows to recover the ordering $\preceq$ by looking at the value $g(0)=t(g)$: $g\succ 1$ if and only if $g(0)>0$. Thus, we can see the dynamical realization as an order-preserving embedding of $G$ into $\hr$ with some standard ordering $\preceq_{\mathbf{x},\Omega}$ with $x_1=0$ and $\Omega(1)=+$. This highlights the importance of standard orderings beyond the proof of the orderability of $\hr$.

	\section{Typical orderings}

\label{sec:proof1}

\begin{proposition}\label{gentyp}
Every general dynamical-lexicographic ordering is typical.
\end{proposition}

\begin{proof}
Let $\preceq=\preceq_{\leq_w,\Omega}$ be a general dynamical-lexicographic ordering, and let $f\in\hr$ be a non-trivial function.

The sign of $f$ is determined at some point $x\in\R$. In particular, this means that $f(y)=y$ for all $y<_w x$, while $f(x)\neq x$. Equivalently, $y\notin A_f\cup B_f$ for all $y<_w x$, and $x\in A_f\cup B_f$.

Now take $g\in\hr$ with $A_g=A_f$ and $B_g=B_f$. Then $y\notin A_g\cup B_g$ for all $y<_w x$, and $x\in A_g\cup B_g$. Moreover, $g(x)>x$ if and only if $f(x)>x$. Thus, $g$ and $f$ have the same sign with respect to $\preceq$. Hence $\preceq$ is typical.
\end{proof}

In particular, Proposition \ref{gentyp} implies that every standard ordering is typical.

\begin{lemma}\label{tclosed}
The set $\mathcal{T}$ of all typical orderings is closed.
\end{lemma}

\begin{proof}
We prove that its complement is open. Let $\preceq$ be a non-typical ordering. Then there exist functions $f,g \in \hr$ with $A_f=A_g$ and $B_f=B_g$ such that 
\begin{equation}\label{f1g}
f \prec 1 \prec g.
\end{equation}
Consider the set $U_{f,g^{-1}}$ of all orderings satisfying \eqref{f1g}. By definition of the topology on $\LO(\hr)$, this set is open. Moreover, it contains no typical orderings, since every ordering satisfying \eqref{f1g} is non-typical. Thus, $U_{f,g^{-1}}$ is an open neighborhood of $\preceq$ contained in $\LO(\hr)\setminus\mathcal{T}$. Hence $\LO(\hr)\setminus\mathcal{T}$ is open, and $\mathcal{T}$ is closed.
\end{proof}

\begin{lemma}\label{AnB}
Let $f_1,\dots,f_n \in \hr$ be positive functions with respect to some typical ordering $\preceq$. Denote $A_i=A_{f_i}$ and $B_i=B_{f_i}$. Then
\[
\bigcup_{i=1}^n A_i \;\neq\; \bigcup_{i=1}^n B_i.
\]
\end{lemma}

\begin{proof}
Assume, for contradiction, that
\begin{equation}\label{l1}
\bigcup_{i=1}^n A_i = \bigcup_{i=1}^n B_i,
\end{equation}
and denote this common set by $A$.

We aim to construct functions $g_1,\dots,g_n,h_1,\dots,h_n \in \hr$ with 
\[
A_i = A_{g_i} = A_{h_i}, \qquad B_i = B_{g_i} = B_{h_i},
\]
such that $g \coloneq g_1 \circ \dots \circ g_n$ and $h \coloneq h_1 \circ \dots \circ h_n$ satisfy
\[
A_g = B_h = A, \qquad B_g = A_h = \emptyset.
\]
Under these conditions, each $g_i$ and $h_i$ has the same sign as $f_i$ in the typical ordering $\preceq$, and so all $g_i,h_i$ are positive. Consequently, both $g$ and $h$ are positive as products of positive elements. This contradicts \eqref{l1}, since $g$ and $h$ must have opposite signs: indeed $A_g=A_{h^{-1}}=B_h$ and $B_g=B_{h^{-1}}=A_h$.

We decompose each $f_i$ as
\[
f_i = f_i^{+} \circ f_i^{-} = f_i^{-} \circ f_i^{+},
\]
where
\[
f_i^{+}(x) = \max\{f_i(x),x\}, \qquad f_i^{-}(x) = \min\{f_i(x),x\}.
\]
Note that $f_i^+,f_i^- \in \hr$ with $A_{f_i^{+}}=A_i$, $B_{f_i^{+}}=\emptyset$, while $A_{f_i^{-}}=\emptyset$, $B_{f_i^{-}}=B_i$.

Define
\begin{equation}\label{f}
f \coloneq f_1 \circ \dots \circ f_n
= f_1^+ \circ f_1^- \circ f_2^+ \circ f_2^- \circ \dots \circ f_n^+ \circ f_n^-,
\end{equation}
and
\[
f^+ \coloneq f_1^+ \circ \dots \circ f_n^+, \qquad 
f^- \coloneq f_1^- \circ \dots \circ f_n^-.
\]

We construct $g_i=f_i^+\circ g_i^-$, where
\begin{equation}\label{gminus}
g_i^-(x)=\max\{f_i^-(x),\gamma(x)\},
\end{equation}
for some suitable $\gamma\in\hr$ with $A_\gamma=\emptyset$ and $B_\gamma=A$. This would result in $A_{g^-_i}=A_{f^-_i}=\emptyset$ and $B_{g^-_i}=B_{f^-_i}=B_i$.

Consider the auxiliary function
\[
\Theta(x,t) \coloneq f_1^+\big(\dots f_{n-1}^+\big(f_n^+(x-t)-t\big)\dots -t\big).
\]
In other words, $\Theta(x,t)$ is obtained from $f$ in \eqref{f} by replacing each occurrence of $f_i^-$ by the translation $x \mapsto x-t$. For fixed $x$, the function $\Theta(x,\cdot)$ is strictly decreasing and continuous, with $\Theta(x,0) = f^+(x) \geq x$ and $\lim_{t\to\infty} \Theta(x,t) = -\infty$. Hence, there exists a unique $t_x \geq 0$ such that $\Theta(x,t_x)=x$.

Suppose $g$ is constructed and fix $x \in A$. Consider the sequence
\[
x_0 = x, \; x_1 = g_n^{-}(x_0), \; x_2 = f_n^+(x_1), \; x_3 = g_{n-1}^-(x_2), \; \dots,  \; x_{2n} = f_1^+(x_{2n-1}) = g(x).
\]
We have $x_{i+1} \geq x_i$ for odd $i$, and $x_{i+1} \leq x_i$ for even $i$. In the latter case,
\[
x_i - x_{i+1}= x_i-g_{n+1-i/2}(x_i) \leq x_i - \gamma(x_i)
\]
since $g^-_{n+1-i/2}(x_i)\geq \gamma(x_i)$ by \eqref{gminus}.

We can bound from above each element of this sequence by ignoring all decreasing actions $g_j$:
\begin{equation*}
    x_i\leq f_k^+ (f_{k+1}^+(\dots f_n^+(x)\dots))\leq f_1^+ (f_2^+(\dots f_n^+(x)\dots))=f^+(x),
\end{equation*}
where $k=n+1-i/2$ for even, and $k=n+1-(i-1)/2$ for odd $i$. 

Similarly, the natural lower bound is 
\[
x_i\geq g^-_1(g^-_2( \dots  g_n^-(x) \dots)).
\]
Combining this with the fact $g^-_j(y)\geq f^-_j(y)$ for all $y\in\R$, we get the bound
\begin{equation*}
    x_i\geq  f_1^- (f_2^-(\dots f_n^-(x)\dots))=f^-(x).
\end{equation*}

Now, since $x_i \in [f^-(x), f^+(x)]$, and $g(x_i)\geq \gamma(x_i)$, each action by $g^-_i$ in 
\[
g(x)=f^+_1\circ g^-_1\circ\dots\circ f^+_n \circ g^-_n (x)
\]
decreases the value by at most $\sup\{y-\gamma(y) \mid y \in [f^-(x),f^+(x)]\}$.

It follows then that
\[
g(x) \geq \Theta\big(x, \sup\{y-\gamma(y) \mid y \in [f^-(x),f^+(x)]\}\big).
\]
Thus, to guarantee $g(x)>x$, it suffices to require
\begin{equation}\label{gammacond}
\sup\{y-\gamma(y) \mid y \in [f^-(x),f^+(x)]\} < t_x.
\end{equation}

Define
\[
\gamma_0(x) \coloneq x - \frac12 \min\{t_y \mid y \in [(f^+)^{-1}(x), (f^-)^{-1}(x)]\},
\]
where the minimum exists since the mapping $y \mapsto t_y$ is continuous. For $y \in [f^-(x),f^+(x)]$, we obtain
\begin{align*}
y-\gamma_0(y) 
&= \frac12 \min\{t_z \mid z \in [(f^+)^{-1}(y),(f^-)^{-1}(y)]\} \\
&= \frac12 \min\{t_z \mid y \in [f^-(z), f^+(z)]\} \\
&\leq \frac12 t_x < t_x.
\end{align*}
Hence, $\gamma_0$ satisfies \eqref{gammacond}. Note that $\gamma_0$ is continuous but not necessarily increasing, so it may fail to belong to $\hr$. To remedy this, set
\[
\gamma(x) \coloneq \frac12 \Big(x + \sup_{y \leq x} \gamma_0(y)\Big).
\]
Note that the function $x\mapsto \sup_{y \leq x} \gamma_0(y)$ is non-decreasing, hence, $\gamma$ is increasing.
Clearly $\gamma$ is continuous, so $\gamma \in \hr$. Moreover, since $\gamma_0(x) \leq x$, we have $\gamma(x)\leq x$ for all $x$, and thus $A_\gamma = \emptyset$.

For $x \notin A$, the interval $[f^-(x),f^+(x)]$ degenerates to a singleton $\{x\}$ and $t_x=0$, so $\gamma(x)=x$. On the other hand, if $x \in A$, then $[f^-(x),f^+(x)]$ is non-degenerate and $t_y > 0$ for all $y$ in this interval, hence $\gamma_0(x) < x$. By continuity,
\begin{equation}\label{closex}
\gamma_0(y) < y - \frac12(x-\gamma_0(x)) 
< \frac12(x+\gamma_0(x))
\end{equation}
for all $y$ in some left neighborhood $(y_0,x]$ of $x$. For $y \leq y_0$, monotonicity of $\gamma$ gives
\begin{equation}\label{farx}
\gamma(y) \leq \gamma(y_0) \leq y_0.
\end{equation}
Together, \eqref{closex} and \eqref{farx} imply $\gamma(x)<x$, and hence $B_\gamma = A$.

Thus, the functions $g_i^-$ defined by \eqref{gminus} satisfy $A_{g_i^-}=\emptyset$ and $B_{g_i^-}=B_i$. Consequently, $g_i = f_i^+ \circ g_i^-$ satisfies $A_{g_i}=A_i$ and $B_{g_i}=B_i$. The composition $g=g_1\circ\dots\circ g_n$ then satisfies $A_g=A$ and $B_g=\emptyset$.

A similar construction yields a function $h$ with $B_h=A$ and $A_h=\emptyset$. 
\end{proof}

    Now we are ready to prove Theorem \ref{S=T}.
	
\begin{proof}[Proof of Theorem \ref{S=T}]
The inclusion
\begin{equation*}
	\overline{\mathcal{S}} \subset \mathcal{T}
\end{equation*}
follows immediately from Lemma \ref{tclosed} together with the fact that every standard ordering is typical.

It remains to prove the reverse inclusion
\begin{equation*}
	\mathcal{T} \subset \overline{\mathcal{S}}.
\end{equation*}

Let $\preceq \in \mathcal{T}$ be a typical ordering, and let $f_1,\dots,f_n \succ 1$ be a finite set of positive elements. We will construct a standard ordering that agrees with $\preceq$ on this set.

By Lemma \ref{AnB}, we have
\[
	\bigcup_{i=1}^n A_{f_i} \;\neq\; \bigcup_{i=1}^n B_{f_i}.
\]
Hence, there exists a point 
\[
	x_1 \in \bigcup_{i=1}^n A_{f_i} \,\triangle\, \bigcup_{i=1}^n B_{f_i}.
\]
After reindexing if necessary, we may assume that $x_1 \in A_{f_1}$ (the case $x_1 \in B_{f_1}$ is analogous).  

Proceeding inductively, for $k=2,\dots,n$ we choose
\[
	x_k \in \bigcup_{i=k}^n A_{f_i} \,\triangle\, \bigcup_{i=k}^n B_{f_i},
\]
and reorder the indices $k,\dots,n$ so that $x_k \in A_{f_k}$ or $x_k \in B_{f_k}$.

We now extend the finite sequence $(x_1,\dots,x_n)$ arbitrarily to a dense sequence $\mathbf{x}=(x_m)_{m\in\mathbb{N}}$ in $\mathbb{R}$.

Define $\Omega:\mathbb{N}\to\{+,-\}$ by setting
\[
	\Omega(k) = 
	\begin{cases}
		+ & \text{if } x_k \in A_{f_k}, \\
		- & \text{if } x_k \in B_{f_k},
	\end{cases}
	\quad\text{for } 1 \leq k \leq n,
\]
and choosing $\Omega(k)$ arbitrarily for $k>n$.

By construction, the standard ordering $\preceq_{\mathbf{x},\Omega}$ makes each $f_i$ positive:
\[
	f_1,\dots,f_n \succ_{\mathbf{x},\Omega} 1.
\]
Thus $\preceq_{\mathbf{x},\Omega}$ agrees with $\preceq$ on the given finite set.

Since the choice of $f_1,\dots,f_n$ was arbitrary, it follows that $\preceq$ belongs to the closure $\overline{\mathcal{S}}$.

Therefore, $\mathcal{T}\subset \overline{\mathcal{S}}$, and the theorem follows.
\end{proof}

	\section{Examples of non-dynamical-lexicographic and non-typical orderings}

\label{sec:proof2}

Before completing the proof of Theorem \ref{hierarchy}, we need to establish some auxiliary results.  

The following Compactness Principle \cite{DNR}*{Theorem 1.44} is a classical fact in the theory of left-orderable groups.  

\begin{theorem}\label{compactnessargument}
A group $G$ is left-orderable if and only if each of its finitely generated subgroups is left-orderable.
\end{theorem}

We will use a slightly modified version of this statement.  

\begin{theorem}\label{comp2}
Let $F\subset G$ be a finite set. Suppose that for every finitely generated subgroup $H<G$, there exists an ordering $\preceq_H$ such that all elements of $F\cap H$ are positive. Then $G$ admits an ordering in which every element of $F$ is positive.
\end{theorem}

\begin{proof}
Recall that a positive cone on $G$ is a subset $P\subset G$ satisfying:  
\begin{itemize}
    \item $P\cdot P\subset P$;
    \item $P\cup P^{-1}=G\setminus\{1\}$.
\end{itemize}
Each positive cone $P$ corresponds bijectively to a left-ordering $\preceq_P$ defined by $f\succ_P g$ if and only if $g^{-1}f\in P$.

Let $F=\{f_1,\dots,f_n\}\subset G$ be finite. By assumption, every finitely generated subgroup $H<G$ admits a left-ordering in which all elements of $F\cap H$ are positive. Define
\[
	\mathfrak{P}_{H,F}=\{\,P\subset G : F\subset P,\; P\cap H \text{ is a positive cone in } H\,\}.
\]

Each $\mathfrak{P}_{H,F}$ is a set from the standard basis of the Tychonoff topology on the power set $2^G$. Hence, it is closed.

Now consider an element 
\[
	P\in \bigcap_{H} \mathfrak{P}_{H,F},
\]
where the intersection ranges over all finitely generated subgroups $H<G$. Take $f,g\in P$ and set $H=\langle f,g\rangle$. Since $P\cap H$ is a positive cone in $H$, we have $fg\in P$. Likewise, for any $g\neq 1$, considering $H=\langle g\rangle$ shows that exactly one of $g$ or $g^{-1}$ belongs to $P$. Hence $P$ is a positive cone on $G$, and clearly $F\subset P$. The corresponding ordering $\preceq_P$ satisfies the claim.

It remains to show that the intersection $\bigcap_{H}\mathfrak{P}_{H,F}$ is nonempty. For this, observe that the family
\[
	\{\mathfrak{P}_{H,F} : H\text{ a finitely generated subgroup of }G\}
\]
has the Finite Intersection Property. Indeed, for finitely many subgroups $H_1,\dots,H_m$ of $G$ we have
\[
	\bigcap_{i=1}^m \mathfrak{P}_{H_i,F} \supset \mathfrak{P}_{H,F},
\]
where $H=\langle H_1,\dots,H_m\rangle$ is the smallest subgroup of $G$ that contains each of $H_i$'s. Since $H$ is finitely generated, $\mathfrak{P}_{H,F}$ is nonempty by assumption. It follows then from Tychonoff’s Theorem that the intersection $\bigcap_{H}\mathfrak{P}_{H,F}$ is nonempty as well. This completes the proof.
\end{proof}

\begin{proposition}\label{geinf}
The group $\Ge$ of germs at $\infty$ of homeomorphisms of $\R$ is left-orderable. Moreover, there exist functions $f,g\in\hr$ with $A_f=A_g=\R$ and an ordering $\preceq^{\prime}_{\infty}$ on $\Ge$ such that
\[
	f_\infty \prec^{\prime}_{\infty} \id \prec^{\prime}_{\infty} g_\infty,
\]
where $f_{\infty}$ and $g_{\infty}$ are the germs at infinity of $f$ and $g$, respectively.
\end{proposition}

\begin{proof}
The orderability of $\Ge$ is established in \cite{Mann}*{Proposition 2.2} (see also \cite{DNR}*{Remark 1.1.13}) using Theorem \ref{compactnessargument}. 

 We present an alternative proof here.

    Fix an arbitrary non-principal ultrafilter $\mathcal{U}$ on $\N$, and equip the set $\mathcal{X}_{\infty}=\{\mathbf{x}_i=\{(x_n^{(i)})\}\}_{i\in I}$ of all increasing to infinity sequences of real numbers with some well-ordering.

    For a non-trivial germ $h_{\infty}\in\Ge$ of $h\in\hr$ let $i=i_h$ be the smallest index, such that the set 
    \[     
    S^{(0)}_{h,i}=\{n\in\N\mid h(x_n^{(i)})=x_n^{(i)}\}     
    \]
    does not belong to the ultrafilter $\mathcal{U}$. 
    This index always exists because one can construct a sequence $\mathbf{x}_i\in \mathcal{X}{\infty}$ with $C_{h,i}=\emptyset$. Moreover, the index $i_h$ does not depend on the choice of the representative $h$.

    Now, we declare the germ $h_{\infty}$ to be positive if the set 
    \[
    S^{+}_{h,i}=\{n\in\N\mid h(x_n^{(i)})>x_n^{(i)}\} =\{n\in \N\mid x_n^{(i)}\in A_h\}
    \]
    belong to the ultrafilter $\mathcal{U}$. Similarly, negative, if it is the set 
    \[
    S^{-}_{h,i}=\{n\in\N\mid h(x_n^{(i)})<x_n^{(i)}\} =\{n\in \N\mid x_n^{(i)}\in B_h\}.
    \]
    Again, this is independent on the specific choice of $h$.

    To verify that this construction indeed defines an ordering on $\Ge$, we need to check that the product of positive germs remains positive. Let $h_{\infty}^{(1)}, h_{\infty}^{(2)}$ be positive germs, and let $h_1,h_2\in\hr$ be their representatives. Let also $h=h_1h_2$, $i_1=i_{h_{\infty}^{(1)}},i_2=i_{h_{\infty}^{(2)}}$, and $i=\min\{i_1,i_2\}$.

    Then, for any index $j<i$ we have
    \[
    S^{(0)}_{h_1,j}\bigcap S^{(0)}_{h_2,j} \in \mathcal{U}
    \]
    as an intersection of two elements of the ultarfilter. But for any $n\in S^{(0)}_{h_1,j}\cap S^{(0)}_{h_2,j}$
    \[
    h(x_n^{(j)})=h_1\bigl(h_2(x_n^{(j)})\bigr)=h_1(x_n^{(j)})=x_n^{(j)}.
    \]
    Hence, $n\in S^{(0)}_{h,j}$, and 
    \[
    S^{(0)}_{h,j}\supset S^{(0)}_{h_1,j}\bigcap S^{(0)}_{h_2,j}.
    \]
    So, $S^{(0)}_{h,j}\in \mathcal{U}$ and $i_h\geq i$.

    Note that $S^{+}_{h_1,i}\in U$ or $S^{+}_{h_2,i}\in U$. So,
    \[
    \left(S^{+}_{h_1,i}\bigcap S^{+}_{h_2,i}\right)\bigcup\left(S^{+}_{h_1,i}\bigcap S^{(0)}_{h_2,i}\right)\bigcup \left(S^{(0)}_{h_1,i}\bigcap S^{+}_{h_2,i}\right)\in \mathcal{U}
    \]
    because in one of these parentheses, both sets are from $\mathcal{U}$. But 
    \[
    \left(S^{+}_{h_1,i}\bigcap S^{+}_{h_2,i}\right)\bigcup\left(S^{+}_{h_1,i}\bigcap S^{(0)}_{h_2,i}\right)\bigcup \left(S^{(0)}_{h_1,i}\bigcap S^{+}_{h_2,i}\right)\subset S^{+}_{h,i}.
    \]
    So $S^{+}_{h,i}\in \mathcal{U}$, $i_h=i$ and $h$ represents a positive germ.

Observe that the germ of any function $h$ with $A_h=\R$ is positive in the ordering $\preceq_{\infty}$ constructed above.

Consider the functions
\[
	f(x)=x+1, \qquad 
	g(x)=
	\begin{cases}
		2x-1, & \text{if } x>1,\\
		x+1, & \text{otherwise}.
	\end{cases}
\]
Note that $A_f=A_g=\R$. Furthermore,
\[
	g(x)=2x-1 > f^{(n)}(x)=x+n
\]
for all sufficiently large $x$. Thus, for any $n\geq 1$,
\begin{equation}\label{inford}
	g_{\infty}\succ_{\infty} f^n_{\infty}\succ_{\infty}\id.
\end{equation}

Let $H$ be a finitely generated subgroup of $\Ge$, and let $\rho$ denote the dynamical realization of $H$ with respect to the restriction of $\preceq_{\infty}$. If $H$ contains both $f_{\infty}$ and $g_{\infty}$, then from \eqref{inford} we obtain
\[
	\rho(g_{\infty})(0)>\rho(f^n_{\infty})(0)>0.
\]
Passing to the limit $n\to\infty$, we deduce
\[
	\rho(g_{\infty})(0)\geq \lim_{n\to\infty}\rho(f^n_{\infty})(0)=:L>0.
\]
In particular, $L$ is finite and
\[
	\rho(g_{\infty})(L)>\rho(g_{\infty})(0)\geq\rho(f_{\infty})(L)=L.
\]

Equip $\hr$ with the standard ordering $\preceq_{\mathbf{x},\Omega}$, where $\mathbf{x}$ is a dense sequence beginning with $x_1=L$, $x_2=0$, and $\Omega$ is a sign function defined by $\Omega(1)=-$, $\Omega(2)=+$. Then
\[
	\rho(g_{\infty}) \prec_{\mathbf{x},\Omega} \id \prec_{\mathbf{x},\Omega} \rho(f_{\infty}).
\]
This induces an ordering $\preceq$ on $H$ with
\[
	g_{\infty} \prec \id \prec f_{\infty}.
\]

Finally, Proposition \ref{geinf} follows from Theorem \ref{comp2} applied with $G=\Ge$ and $F=\{f_{\infty}, g_{\infty}^{-1}\}$.
\end{proof}

Now we are ready to finish the proof of Theorem \ref{hierarchy}.
    
	\begin{proof}[Proof of Theorem \ref{hierarchy}]

It remains to show the strictness of the inclusions
\[
\mathcal{G} \subsetneq \mathcal{T} \subsetneq \LO(\hr).
\]

First, we justify the inclusion
\[
\mathcal{G} \subsetneq \mathcal{T} 
\]
by constructing a typical but not general dynamical-lexicographic ordering.

	Let $\Ho<\hr$ denote the subgroup consisting of homeomorphisms with trivial germs at $\infty$.  
	In particular, we have the identification
	\[
		\faktor{\hr}{\Ho}\simeq \Ge.
	\]

We equip $\Ge$ with the ordering $\preceq_{\infty}$ from the first part of the proof of Proposition \ref{geinf}, and $\Ho$ with any typical ordering $\preceq_0$.

We define the ordering $\preceq$ on $\hr$ by declaring
\[
		h \succ \id \quad \Longleftrightarrow \quad 
		\big(h_{\infty}\succ_{\infty}\id \text{ in }\Ge\big)\ \text{or}\ \big(h\in\Ho \text{ and } h\succ_0 \id\big).
\]

Once the ultrafilter $\mathcal{U}$, and the well-ordering on $\mathcal{X}{\infty}$ are fixed, the sign of any given $f\in \hr\setminus \Ho$ is determined by the sets $A_f$ and $B_f$, namely by whether the set $S^{+}_{f,i}=\{n\in \N\mid x_n^{(i)}\in A_f\}$ or the set $S^{-}_{f,i}=\{n\in \N\mid x_n^{(i)}\in B_f\}$ belongs to the ultrafilter. Hence, the ordering $\preceq$ is typical. 

To see that it is not dynamical-lexicographic, let $f_x^+,f_x^-$ be two functions with the same non-trivial germ at $\infty$, such that
\[
f_x^-(x)<x<f_x^+(x).
\]
Then they have the same sign in $\preceq$, but the opposite signs in any ordering $\preceq_{\leq_w,\Omega}$ with $x$ being the $\leq_w$-smallest point. Since $x$ is arbitrary, this shows that the ordering $\preceq$ differs from any dynamical-lexicographic ordering $\preceq_{\leq_w,\Omega}$.

Finally, we cover the inclusion
\[
\mathcal{T} \subsetneq \LO(\hr).
\]

	Consider the functions $f,g\in\hr$ from Proposition \ref{geinf}, and their germs at $\infty$, $f_{\infty}, g_{\infty}\in\Ge$.
	Then, there exists an ordering $\preceq_{\infty}^{\prime}$ on $\Ge$ such that
	\[
		g_{\infty}\prec^{\prime}_{\infty} \id \prec^{\prime}_{\infty} f_{\infty}.
	\]
	
	As above, we extend the ordering $\preceq_0$ on $\Ho$ to an ordering $\preceq^{\prime}$ on $\hr$ by declaring
	\[
		h \succ^{\prime} \id \quad \Longleftrightarrow \quad 
		\big(h_{\infty}\succ^{\prime}_{\infty}\id \text{ in }\Ge\big)\ \text{or}\ \big(h\in\Ho \text{ and } h\succ_0 \id\big).
	\]
	
	Under this ordering, we have
	\[
		g \prec^{\prime} \id \prec^{\prime} f.
	\]
	Recall that $A_f = A_g = \R$ and $B_f=B_g=\emptyset$.  
	Thus, the ordering $\preceq^{\prime}$ cannot be typical, which concludes the proof.
	
\end{proof}

\end{document}